\documentclass[12pt, a4paper]{article}

\usepackage[utf8]{inputenc}
\usepackage[T1]{fontenc}
\usepackage[english]{babel}
\usepackage{indentfirst}
\usepackage[top=2.5cm,bottom=2.5cm,right=2cm,left=2cm]{geometry}
\usepackage{amsmath}
\usepackage{amsfonts}
\usepackage{amssymb}
\usepackage{amsthm}
\usepackage{graphicx}
\usepackage{yhmath}
\usepackage{babelbib}
\usepackage{lmodern} \normalfont
\usepackage{color}
\usepackage{dsfont}
\usepackage[displaymath, mathlines, pagewise]{lineno}
\usepackage{fancyhdr}
\usepackage{mathtools}
\usepackage{scalerel,stackengine}
\stackMath
\newcommand\reallywidehat[1]{%
\savestack{\tmpbox}{\stretchto{%
  \scaleto{%
    \scalerel*[\widthof{\ensuremath{#1}}]{\kern-.6pt\bigwedge\kern-.6pt}%
    {\rule[-\textheight/2]{1ex}{\textheight}}
  }{\textheight}%
}{0.5ex}}%
\stackon[1pt]{#1}{\tmpbox}%
}

\newtheorem{thm}{Theorem}[section]
\newtheorem{prop}[thm]{Proposition}
\newtheorem{lem}[thm]{Lemma}
\newtheorem{cor}[thm]{Corollary}
\newtheorem{assump}[thm]{Assumption}

\newcommand{\R}{\mathbb{R}}
\newcommand{\C}{\mathbb{C}}

\newcommand{\eps}{\varepsilon}
\title{Convergence and Stability Analysis of the Extended Infinite Horizon Model Predictive Control}

\date{\today}

\author{L.~A.~Alvarez\thanks{School of Chemical Engineering, University of Campinas--UNICAMP, av.~Albert Einstein 500, 13083--852, Campinas SP, Brazil; e-mail: luzadri@unicamp.br} \and D.~F. de Bernardini\footnotemark[2]  \and C.~Gallesco\thanks{Department of Statistics, Institute of Mathematics, Statistics and Scientific Computation, University of Campinas--UNICAMP, rua Sérgio Buarque de Holanda 651, 13083--859, Campinas SP, Brazil; e-mails: $\{$diegofb, gallesco$\}$@unicamp.br}}

\begin{document}
	
\maketitle

\begin{abstract}
	\footnotesize

Model Predictive Control (MPC) is a popular technology to operate industrial systems. It refers to a class of control algorithms that use an explicit model of the system to obtain the control action by minimizing a cost function. At each time step, MPC solves an optimization problem that minimizes the future deviation of the outputs which are calculated from the model. The solution of the optimization problem is a sequence of control inputs, the first input is applied to the system, and the optimization process is repeated at subsequent time steps. In the context of MPC, convergence and stability are fundamental issues. A common approach to obtain MPC stability is by setting the prediction horizon as infinite. For stable open-loop systems, the infinite horizon can be reduced to a finite horizon MPC with a terminal weight computed through the solution of a Lyapunov equation. This paper presents a rigorous analysis of convergence and stability of the extended nominally stable MPC developed by Odloak [Odloak, D. Extended robust model predictive control, AIChE J. 50 (8) (2004) 1824–1836] and the stable MPC with zone control [González, A.H., Odloak, D. A stable MPC with zone control, J. Proc. Cont. 19 (2009) 110-122]. The mathematical proofs consider that the system is represented by a general gain matrix $D_0$, i.e., not necessarily regular, and they are developed for any input horizon $m$. The proofs are based on elementary geometric and algebraic tools and we believe that they can be adapted to the derived MPC approaches, as well as future studies.

\vspace{0.3cm}

\end{abstract}

\section{Introduction}

Model Predictive Control (MPC) originated in the late seventies and has been developed considerably since then (Camacho and Bordons, 2007). It was originally developed to meet the specialized control needs of petroleum refineries and power plants, but today MPC represents a powerful technology to operate complex dynamic systems, with several industrial applications, including process control, automotive systems, robotics, and energy management. MPC refers to a class of control algorithms that use an explicit model of the system to obtain the control action by minimizing a cost function. The model represents a dynamic relation between system inputs (control actions) and outputs (measurements). The purpose of the model is to predict the future response of the outputs over a prediction horizon. At each time step, an MPC algorithm solves an optimization problem that contains a performance cost function, the predictive model of the system and constraints on inputs and outputs. The solution of this problem is a sequence of inputs, the first input in the optimal sequence is then applied to the system, and the optimization process is repeated at the next time step, incorporating updated output measurements.

Fundamental aspects of any control system are convergence and stability. Convergence refers to the ability of the optimization process to reach a solution that satisfies the control objectives within a finite number of iterations. Guarantee of stability is essential to prevent undesirable behaviors such as oscillations, instability, or divergence, which can compromise the performance and safety of the closed-loop system. Various approaches have been proposed to analyze and ensure the stability of MPC. Since MPC uses a prediction model, the stability of the closed-loop system with MPC is classified in two types: if the prediction model perfectly represents the process system, the stability is nominal, and if there is uncertainty in the prediction model, the stability is robust. The nominal stability of MPC can be approached by different methods. In the classical method, stability is forced through indirect methods such as those based on the existence of a Lyapunov function that represents the closed-loop behavior.
 
According to Keerthi and Gilbert (1988), the MPC stability can be obtained when the terminal state is constrained to the origin. If the constrained optimization problem remains feasible along the time steps, the cost function decreases and can be interpreted as a Lyapunov function, then it can be proven that the closed-loop system is stable. The main disadvantage of this method is that the terminal constraint can turn the optimization problem infeasible. 

Michalska and Mayne (1993) proposed the Dual MPC, which is an extension of the terminal state technique to produce a stable MPC. This approach considers a terminal set constraint instead of a terminal state constraint. At the end of the prediction horizon the terminal state is forced to lie in an invariant control set that contains the desired equilibrium state. It means that once the state enters the terminal set, a linear controller will maintain the state inside the terminal set. So, the Dual MPC uses two control laws, one outside the terminal set and one inside the terminal set. 

An usual method to obtain stability of an MPC closed-loop system is to adopt an infinite prediction horizon. In this sense, Rawlings and Muske (1993) developed an MPC regulator with infinite prediction horizon and input and output constraints. For stable open-loop systems, the infinite horizon can be reduced to a finite horizon MPC with a terminal weight computed through the solution of a Lyapunov equation. This MPC has recursive feasibility, that is, if the optimization problem is feasible at time $k$, it will remain feasible at any subsequent time step $k+1$, $k+2$, $\dots$ It was also shown that under these conditions the MPC cost function is strictly decreasing and behaves as a Lyapunov function of the closed-loop system. 

The infinite horizon MPC was later extended to the reference tracking problem by Rodrigues and Odloak (2003). This work overcame the need to know the system steady state allowing the application to the output-tracking problem and the regulator problem with unknown disturbances, which was one of the major barriers to implement infinite horizon MPC in practice. Furthermore, Odloak (2004) presented an easier and more practical infinite horizon MPC. Following the method of Rodrigues and Odloak (2003), slack variables were added to the optimization problem, allowing a minimal violation of constraints, and keeping the cost function limited for the disturbed system. This feature is important for practical implementation.

In this paper, we consider the MPC proposed by Odloak (2004), which is an infinite horizon MPC that considers an Output Prediction-Oriented Model (OPOM). The OPOM model is a state-space model arranged in the incremental form of inputs, and it is developed from the analytical form of the step response of the system (Rodrigues and Odloak, 2003). For this type of model, the output steady state prediction is one of the process states, which is suitable to impose the end constraint in the infinite horizon MPC. The MPC proposed by Odloak (2004) considers open-loop stable systems. Over the last 20 years, this controller has gained attention in the academic community and chemical process industry. The infinite horizon MPC with OPOM model was further developed for zone control, where the outputs are controlled inside zones or ranges instead of fixed set-points (Gonzalez and Odloak, 2009). The consideration of the output zones results in additional degrees of freedom left to the MPC optimization problem, which means that all or some inputs are free to be moved to optimal targets that can be defined externally or calculated through a real time optimization (RTO) layer. The MPC with zone control served as a basis for subsequent approaches: it was extended to systems with integrating poles (Carrapiço and Odloak, 2005; Gonzalez et al, 2007; Costa et al, 2021), to integrating systems with optimizing targets (Alvarez et al., 2009), to dead time systems (Gonzalez and Odloak, 2011; Santoro and Odloak, 2012; Martins et al, 2013; Pataro et al, 2019, 2022) and to unstable systems (Martins and Odloak, 2016). The infinite horizon MPC with OPOM model and zone control was also formulated in two layers to receive the real time optimization targets, which are the solution of an economic optimization problem (Alvarez and Odloak, 2010, 2014; Oliveira et al, 2019). Furthermore, the infinite horizon MPC with OPOM model has been implemented in the context of process design methodologies that consider simultaneously economic profit, dynamic performance, and process safety (Carvalho and Alvarez, 2020; Marques and Alvarez, 2023).

Additionally, control technologies based on the infinite horizon MPC with OPOM model have been successfully applied in real process industries, especially oil refineries (Carrapiço et al, 2009; Porfirio and Odloak, 2011; Strutzel et al, 2013; Strutzel, 2014; Martin et al, 2019) as well as pilot scale plants (Martin et al, 2013; Silva et al., 2020). Nowadays, the MPC with OPOM model is part of an in-house advanced control package developed by Petrobras (Petróleo Brasileiro S.A.) and has been implemented in many process units of the main oil refineries of Brazil (Sencio, 2022).

Considering the relevance of the work developed after Odloak (2004), the objective of this paper is to provide a rigorous proof of convergence and stability for the MPC with OPOM model with a general gain matrix $D_0$, i.e.~not necessarily regular, allowing the input and output vectors to have different dimensions. The original paper provides insights to show recursive feasibility, convergence and stability for the case where the gain matrix $D_0$ is regular and the input horizon $m$ is equal to 1, this last simplification is also present in the subsequent works. Here, we use elementary geometric and algebraic tools to develop mathematical proofs that work for any input horizon $m$ and any gain matrix $D_0$ in the infinite horizon MPC, and we provide explicit expressions that can be implemented in practice.  Furthermore, we developed rigorous proofs of convergence and stability for the MPC with zone control (Gonzalez and Odloak, 2009) considering a general input horizon $m$. We believe that the proofs presented in this work can be adapted to the derived approaches and serve as a mathematical background for future studies and developments in the MPC field.

The paper is organized as follows. In Section 2 we present the OPOM model and the infinite horizon MPC formulation with fixed set-point, the corresponding convergence and stability theorems and their proofs. In Section 3 we present the infinite horizon MPC with zone control, the convergence and stability results and their proofs.

\section{Extended Infinite Horizon MPC}\label{Sec2}
\subsection{Formulation and results}\label{Model}

The extended version of the infinite horizon MPC considers the following OPOM model, which is a discrete time state-space model obtained from the step response (Odloak, 2004; Gonzalez and Odloak, 2009).
For discrete time $k\geq 0$, let
\begin{align}
	\begin{bmatrix}
		x_s(k+1) \\ x_d(k+1)
	\end{bmatrix}
	=
	\begin{bmatrix}
		I  & 0 \\ 0 & F
	\end{bmatrix}
	\begin{bmatrix}
		x_s(k) \\ x_d(k)
	\end{bmatrix}
	+ 
	\begin{bmatrix}
		D_0 \\ D_d
	\end{bmatrix}
	\big(u(k+1) - u(k)\big)
	\label{model1}
\end{align}
and
\begin{align}
y(k) = 
\begin{bmatrix}
I & \Psi 
\end{bmatrix}
\begin{bmatrix}
x_s(k) \\ x_d(k)
\end{bmatrix},
\label{model2}
\end{align}
where $x_s\in \R^{n_y}$, $x_d \in \C^{n_d}$, $u\in \R^{n_u}$, $y\in\R^{n_y}$, $F \in \C^{n_d\times n_d}$, $D_0 \in \R^{n_y\times n_u}$, $D_d \in \C^{n_d\times n_u}$, $\Psi \in \R^{n_y\times n_d}$ and $I$ is the identity matrix of dimension $n_y$. It is worth noting that the matrix $D_d$ considered here corresponds to the matrix $D^dFN$ in Odloak (2004). In the state equation (\ref{model1}), $x_s$ is called the \textit{static part} of the state of the system, while $x_d$ is called the \textit{dynamic part}. The vector $u$ represents the \textit{inputs} of the model and $y$ stands for the \textit{outputs}. Matrix $D_0$ is called the \textit{static gain} of the system.

In the infinite horizon MPC setting, we denote by $m\in\mathbb{N}$ the input horizon. Since, at each time step $k\geq 1$, an optimization problem is solved for $u$ over the horizon time interval of size $m$, we introduce the following notation: for $j=0,1, \dots, m-1$, we define $\Delta u(j|k)$ as the $j$-th move of the input solution of the optimization problem (to be defined below) at time step $k$. For the other variables of the model, we will use a similar notation. It is important to point out that, at each time step $k$, only the first move of the input solution of the optimization problem, $\Delta u(0|k)$, is implemented in the system. As a consequence, at each time step $k\geq 1$, $u(k) = \sum_{\ell=1}^{k} \Delta u(0|\ell)$.

The extended infinite horizon MPC is based on the following cost function
\begin{align*}
V_k := \sum_{j=0}^{\infty} \big[e(j|k)-\delta_k\big]^{T}Q\big[e(j|k)-\delta_k\big] + \sum_{j=0}^{m-1} \Delta u(j|k)^{T}R\Delta u(j|k) + \delta_k^{T}S\delta_k, ~~k\geq 1,
\end{align*}
where $e(j|k) := y(j|k) - r$, $r\in\R^{n_y}$ is the output set-point, $\delta_k\in\R^{n_y}$ is a \textit{slack} vector, $Q\in \R^{n_y\times n_y}$, $R\in \R^{n_u\times n_u}$ and $S\in \R^{n_y\times n_y}$ are positive definite. To prevent the above cost from being unbounded, as discussed in Odloak (2004), the following constraint is imposed
\begin{align}
	x_s(m-1|k) - \delta_k - r = 0, ~~k\geq 1.
	\label{rest1}
\end{align}

We will work under the following assumption which was already present in Odloak (2004).
\begin{assump}
	The system is stable, that is, the spectral radius of $F$ is strictly smaller than~$1$.
	\label{Assump1}
\end{assump}

Using Assumption \ref{Assump1} and (\ref{rest1}), we obtain, for $k\geq 1$,
\begin{align}
\label{Cost}
	V_k &= 
	\begin{multlined}[t]
	\sum_{j=0}^{m-1} \big[e(j|k)-\delta_k\big]^{T}Q\big[e(j|k)-\delta_k\big] + x_d(m-1|k)^{T}\bar{Q}x_d(m-1|k) \nonumber\\ 
	+ \sum_{j=0}^{m-1} \Delta u(j|k)^{T}R\Delta u(j|k) + \delta_k^{T}S\delta_k
	\end{multlined} \nonumber\\
	&= \sum_{j=0}^{m-1} \big\|e(j|k)-\delta_k\big\|_Q^2 + \big\|x_d(m-1|k)\big\|_{\bar{Q}}^2 + \sum_{j=0}^{m-1} \big\|\Delta u(j|k)\big\|_R^2 + \big\|\delta_k\big\|_S^2,
\end{align}
where
\begin{align*}
\bar{Q} := \sum_{j=1}^{\infty} (F^j)^T\Psi^T Q \Psi F^j
\end{align*}
is positive semidefinite. Also, note that $\bar{Q} - F^{T}\bar{Q}F = F^{T}\Psi^{T}Q\Psi F$.

The control optimization problem of the extended infinite horizon MPC can be stated as follows: for all $k\geq 1$,
\begin{align*}
	\min_{\Delta u_k, \delta_k} V_k
\end{align*}
subject to (\ref{rest1}) and 
\begin{align*}
u(k-1) + \sum_{j=0}^i\Delta u(j|k)\in \mathbb{U},\;\text{for}\; 0\leq i< m\:\;\; ~\text{and}~\;\;\; \Delta u(j|k) \in \Delta \mathbb{U}, \text{ for } 0\leq j<m,
\end{align*}
where $\mathbb{U}$ and $\Delta\mathbb{U}$ are fixed rectangles in $\R^{n_u}$ containing the origin and 
\begin{align*}
\Delta u_k = \big[\Delta u(0|k)^T ~\Delta u(1|k)^T ~\cdots ~\Delta u(m-1|k)^T\big]^T \in \R^{mn_u}.
\end{align*}
We assume that, at time $0$, the system is in the steady state $u(0)={\bf 0}$, $x_s(0)={\bf 0}$, $x_d(0)={\bf 0}$. Since $\mathbb{U}$ and $\Delta\mathbb{U}$ are convex sets and $R$ is positive definite, the solution of the above optimization problem is unique. Let $\Delta u_k^*$, $\delta_k^*$ denote the solution at time step $k$, and let $V_k^*$ be the corresponding cost,
\begin{align*}
V_k^* = \sum_{j=0}^{m-1} \big\|e^*(j|k)-\delta_k^*\big\|_Q^2 + \big\|x_d^*(m-1|k)\big\|_{\bar{Q}}^2 + \sum_{j=0}^{m-1} \big\|\Delta u^*(j|k)\big\|_R^2 + \big\|\delta_k^*\big\|_S^2,
\end{align*}
where $e^*$, $x_s^*$ and $x_d^*$ are obtained from $\Delta u_k^*$, (\ref{model1}) and (\ref{model2}).

We also need the following

\begin{assump}
	The output reference $r$ is such that $r=D_0 u_r$ for some $u_r\in \mathbb{U}$. 
	\label{Assump3}
\end{assump}

We finally state the results about convergence and stability for the extended infinite horizon MPC. We emphasize that we do not suppose that the gain matrix $D_0$ is regular in the following theorems.

\begin{thm}[Convergence]
\label{Maintheo}
	Under Assumptions \ref{Assump1} and \ref{Assump3}, we can choose the matrix $S$ such that 
$$\displaystyle\lim_{k\to\infty}V_k^*=0.$$
\end{thm}

\begin{thm}[Stability]
	\label{Maintheo2}
	Under Assumptions \ref{Assump1} and \ref{Assump3}, the controller is stable, that is, for any $\eta>0$, there exists $\varepsilon>0$ such that $\|r\| \leq \varepsilon$ implies that $\|e^*(0|k)\| \leq \eta$ for all $k\geq 1$.
\end{thm}

\subsection{Proofs of Theorems \ref{Maintheo} and \ref{Maintheo2}}\label{Proofs}

We start with several technical results that will be useful to prove Theorems \ref{Maintheo} and \ref{Maintheo2}. The following lemma was already established in Odloak (2004), but for the sake of completeness we present its proof.

\begin{lem}
	Under Assumption \ref{Assump1}, the sequence $(V_k^*)_{k\geq 1}$ is non-increasing.
	\label{lemma1}
\end{lem}

\begin{proof}
At time step $k+1$, let $\Delta \tilde{u}_{k+1}, \tilde{\delta}_{k+1}$ be such that
\begin{align*}
\Delta \tilde{u}(j|k+1) &= \Delta u^*(j+1|k), \text{ for } j=0,1,\dots,m-2, \\
\Delta \tilde{u}(m-1|k+1) &= 0, \\
\tilde{\delta}_{k+1} &= \delta_k^*.
\end{align*}
Note that $\Delta \tilde{u}_{k+1}, \tilde{\delta}_{k+1}$ is a feasible strategy for the control optimization problem at time step $k+1$ (by ``feasible strategy'' we mean that $\Delta \tilde{u}_{k+1}, \tilde{\delta}_{k+1}$ satisfies all the constraints of the optimization problem). Indeed, by letting $e_s(j|k) := x_s(j|k) - r$, we have that
\begin{align*}
e_s^*(0|k) + D_0\sum_{j=0}^{m-1}\Delta \tilde{u}(j|k+1) - \tilde{\delta}_{k+1} &= e_s^*(0|k) + D_0\sum_{j=1}^{m-1}\Delta u^*(j|k) - \delta^*_{k} \\
&= e_s^*(0|k-1) + D_0\sum_{j=0}^{m-1}\Delta u^*(j|k) - \delta_k^* = 0,
\end{align*}
so that (\ref{rest1}) is satisfied. Moreover, note that $\Delta \tilde{u}(j|k+1) \in \Delta\mathbb{U}$ for $0\leq j \leq m-1$, and recalling that $\Delta u^*(m|k)=0$, we obtain
\begin{align*}
u^*(k) + \sum_{j=0}^{i} \Delta \tilde{u}(j|k+1) = u^*(k) + \sum_{j=1}^{i+1} \Delta u^*(j|k) = u^*(k-1) + \sum_{j=0}^{i+1} \Delta u^*(j|k) \in \mathbb{U}
\end{align*}
for $0\leq i \leq m-1$.
 
The cost corresponding to this strategy is
\begin{align*}
\tilde{V}_{k+1} &= \sum_{j=0}^{m-1} \big\|\tilde{e}(j|k+1)-\tilde{\delta}_{k+1}\big\|_Q^2 + \big\|\tilde{x}_d(m-1|k+1)\big\|_{\bar{Q}}^2 + \sum_{j=0}^{m-1} \big\|\Delta \tilde{u}(j|k+1)\big\|_R^2 + \big\|\tilde{\delta}_{k+1}\big\|_S^2 \\
&= \sum_{j=0}^{m-1} \big\|\tilde{e}(j|k+1)-\delta^*_{k}\big\|_Q^2 + \big\|\tilde{x}_d(m-1|k+1)\big\|_{\bar{Q}}^2 + \sum_{j=1}^{m-1} \big\|\Delta u^*(j|k)\big\|_R^2 + \big\|\delta^*_{k}\big\|_S^2
\end{align*}
and, since $\tilde{e}(j|k+1) = e^*(j+1|k)$ for $0\leq j \leq m-2$, we have that
\begin{align*}
\tilde{V}_{k+1} - V_k^* = \big\|\tilde{e}(m-1|k+1)-\delta^*_{k}\big\|_Q^2 -\big\|e^*(0|k)-\delta_k^*\big\|_Q^2 + \big\|\tilde{x}_d(m-1|k+1)\big\|_{\bar{Q}}^2 \\
- \big\|x_d^*(m-1|k)\big\|_{\bar{Q}}^2 - \big\|\Delta u^*(0|k)\big\|_R^2.
\end{align*}
Now observe that 
\begin{align*}
\tilde{e}(m-1|k+1)-\delta^*_{k} = \tilde{e}_s(m-1|k+1) + \Psi\tilde{x}_d(m-1|k+1) - \delta^*_{k} = \Psi\tilde{x}_d(m-1|k+1),
\end{align*}
\begin{align*}
\tilde{x}_d(m-1|k+1) = F \tilde{x}_d(m-2|k+1) + D_d \Delta \tilde{u}(m-1|k+1) = F \tilde{x}_d(m-2|k+1)
\end{align*}
and
\begin{align*}
\tilde{x}_d(m-2|k+1) &= F^{m-1}x_d^*(0|k) + \sum_{i=0}^{m-2} F^{m-2-i} D_d\Delta \tilde{u}(i|k+1) \\
&= F^{m-1}x_d^*(0|k) + \sum_{i=1}^{m-1} F^{m-1-i} D_d\Delta u^*(i|k) \\
&= x_d^*(m-1|k),
\end{align*}
so that
\begin{align*}
\big\|\tilde{e}(m-1|k+1)-\delta^*_{k}\big\|_Q^2 + \big\|\tilde{x}_d(m-1|k+1)\big\|_{\bar{Q}}^2 &= \big\|\Psi Fx_d^*(m-1|k)\big\|_Q^2 + \big\|Fx_d^*(m-1|k)\big\|_{\bar{Q}}^2 \\
&= \big\|x_d^*(m-1|k)\big\|_{(\Psi F)^TQ(\Psi F)+F^T\bar{Q}F}^2 \\
&= \big\|x_d^*(m-1|k)\big\|_{\bar{Q}}^2
\end{align*}
and then
\begin{align*}
\tilde{V}_{k+1} - V_k^* = -\big\|e^*(0|k)-\delta_k^*\big\|_Q^2 - \big\|\Delta u^*(0|k)\big\|_R^2 \leq 0.
\end{align*}

Since the strategy $\Delta \tilde{u}_{k+1}, \tilde{\delta}_{k+1}$ at time step $k+1$ is not necessarily the optimal one, we have that $V_{k+1}^* \leq \tilde{V}_{k+1}$ and the sequence $(V_k^*)_{k\geq 1}$ is non-increasing.
\end{proof}

\begin{lem}
	Under Assumption \ref{Assump1}, we have that
	\begin{equation*}
	x_d^*(0|k) \to 0,~\text{ as } k\to\infty. 
	\end{equation*}
	\label{lemma2}
\end{lem}
\begin{proof}
From the proof of Lemma \ref{lemma1}, we have that
\begin{align*}
V_{k+1}^* - V_k^* \leq -\big\|e^*(0|k)-\delta_k^*\big\|_Q^2 - \big\|\Delta u^*(0|k)\big\|_R^2.
\end{align*}
Observe that $(V_k^*)_{k\geq 1}$ converges since it is non-increasing and non-negative. Therefore we deduce that
\begin{align}
\label{tozero}
\big\|e^*(0|k)-\delta_k^*\big\|_Q^2 \xrightarrow{k\to\infty} 0 ~\text{ and }~ \big\|\Delta u^*(0|k)\big\|_R^2 \xrightarrow{k\to\infty} 0.
\end{align}
Moreover, for $p\geq 1$, we have that
\begin{align*}
x_d^*(0|k+p) = F^{p} x_d^*(0|k) + \sum_{j=0}^{p-1} F^j D_d \Delta u^*(0|k+p-j).
\end{align*}
Now, let $\|\cdot\|$ be the euclidean norm on $\R^{n_d}$. By Assumption \ref{Assump1},  we have that
\begin{align*}
\sum_{j=0}^{\infty} \|F^j\| < \infty.
\end{align*}
Hence, for any $\varepsilon>0$, there exists $k_0\geq 1$ such that
\begin{align*}
\sup_{p\geq 1}\big\|x_d^*(0|k_0+p) - F^{p} x_d^*(0|k_0)\big\| \leq \sup_{j\geq 1}\big\|D_d \Delta u^*(0|k_0+j)\big\|\sum_{j=0}^{\infty} \|F^j\| < \varepsilon/2,
\end{align*}
and, since $\displaystyle\lim_{j\to\infty}F^j=0$, there exists $p_0 \geq 1$ such that, for all $p\geq p_0$, 
\begin{align*}
\big\|F^{p} x_d^*(0|k_0)\big\| < \varepsilon/2.
\end{align*}
Finally, for all $p\geq p_0$,
\begin{align*}
\big\|x_d^*(0|k_0+p)\big\| \leq \big\|F^{p} x_d^*(0|k_0)\big\| + \sup_{p\geq 1}\big\|x_d^*(0|k_0+p) - F^{p} x_d^*(0|k_0)\big\| < \varepsilon,
\end{align*}
that is, $x_d^*(0|k) \to 0$ as $k\to\infty$.
\end{proof}

We now equip $\R^{n_u}$ with the euclidean norm $\|\cdot\|$ and the associated scalar product $\langle \cdot, \cdot\rangle$. We decompose 
$$\R^{n_u}=\ker{D_0}\oplus (\ker{D_0})^{\perp}.$$ 
We also denote by $D_0^{\perp}$ the restriction of $D_0$ to $(\ker{D_0})^{\perp}$ with values in $\text{Im}D_0\subset \R^{n_y}$. Observe that $D_0^{\perp}$ is a linear isomorphism.

 \begin{lem}
	Under Assumption \ref{Assump1} we have that
	\begin{equation*}
	 \Big(\sum_{j=0}^{m-1} \Delta u^*(j|k)\Big)_\perp \to 0, ~\text{ as } k\to\infty. 
	\end{equation*}
	\label{lemma3}
\end{lem}
\begin{proof}
Recall that $e_s(j|k) := x_s(j|k) - r$. Since 
\begin{align*}
e^*(0|k) = e_s^*(0|k) + \Psi x_d^*(0|k)
\end{align*}
and 
\begin{align*}
e^*(0|k)-\delta_k^* \xrightarrow{k\to\infty} 0 ~~\text{ and }~~ x_d^*(0|k) \xrightarrow{k\to\infty} 0,
\end{align*}
we have that
\begin{align*}
e_s^*(0|k)-\delta_k^*\xrightarrow{k\to\infty} 0.
\end{align*}
But
\begin{align*}
e_s^*(0|k) - \delta^*_{k} = -D_0\Big(\sum_{j=1}^{m-1}\Delta u^*(j|k)\Big)_\perp=-D^{\perp}_0\Big(\sum_{j=1}^{m-1}\Delta u^*(j|k)\Big)_\perp
\end{align*}
so that, since $D^\perp_0$ is a linear isomorphism,
\begin{align*}
\Big(\sum_{j=1}^{m-1}\Delta u^*(j|k)\Big)_\perp \xrightarrow{k\to\infty} 0
\end{align*}
and finally, by (\ref{tozero}), we obtain
\begin{align*}
\Big(\sum_{j=0}^{m-1}\Delta u^*(j|k) \Big)_\perp\xrightarrow{k\to\infty} 0.
\end{align*}
\end{proof}

Recall Assumption \ref{Assump3} and  let us denote by $U_r:=u_r+\ker D_0$. Observe that for any $u\in U_r$, $D_0u=r$. Let us also denote by $P_r$ the orthogonal projection on the affine subspace $U_r$. We recall that $P_r$ is not a linear map unless $r=\bf{0}$.

We now build up the matrix $S$ (in the canonical basis of $\R^{n_y}$) that will be considered later on. For this, fix respectively an orthonormal basis of $(\ker D_0)^{\perp}$, $\{v_1,v_2,\dots, v_k\}$, an orthonormal basis of $\text{Im} D_0$, $\{w_1,w_2,\dots, w_k\}$, and an orthonormal basis of $(\text{Im} D_0)^{\perp}$, $\{w_{k+1},w_{k+2},\dots, w_{n_y}\}$. Consider the square matrix $M$ of $D_0^\perp$ in the first two basis. Using the $LQ$-factorization, there exists a lower triangular matrix $L$ and an orthogonal matrix $O$ such that $M=LO$. Since $M$ is regular, we have that $L$ is regular. Now, consider the  matrix $K_1$ (resp.~$K_2$) of dimension $k \times n_y$ (resp.~$(n_y-k) \times n_y$) whose column vectors are the orthogonal projections of the canonical vectors of $\R^{n_y}$ on $\{w_1,w_2,\dots,w_k\}$ (resp.~$\{w_{k+1},w_{k+2},\dots, w_{n_y}\}$). Define $\hat{S}=K_1^T(L^{-1})^T L^{-1}K_1 + K_2^T K_2$, which is positive definite on $\R^{n_y}$. Furthermore, we can check that for all $v\in (\ker D_0)^{\perp}$, $\|D_0v\|^2_{\hat{S}}=(D_0v)^T \hat{S}D_0v=\|v\|^2$. 

\begin{prop}
\label{PropApprox}
Consider $S=\beta \hat{S}$ with $\beta$ a positive real number. We can choose $\beta$ large enough (depending only on the parameters of the model other than $S$) such that under Assumptions \ref{Assump1} and \ref{Assump3} we have that
$$
\limsup_{k\to \infty} \|u^*(0| k)-P_r u^*(0|k)\|=0.
$$
\end{prop}

\begin{proof}
The proof goes by contradiction, that is, assume that $\limsup_{k\to \infty} \|u^*(0| k)-P_r u^*(0|k)\|=c>0$. 
Then, for all $\eps>0$ there exists a subsequence $(k_n)_{n\geq 1}$ such that for all $n\geq 1$,
\begin{equation*}
 \|u^*(0| k_n-1)-P_ru^*(0|k_n-1)\|\geq c-\frac{\eps}{2}.
\end{equation*}
By Lemma \ref{lemma3}, there exists $n_0$ such that, for $n\geq n_0$, 
\begin{equation*}
	\|(u^*(m-1|k_n)-u^*(0|k_n-1))_{\perp}\|\leq \frac{\eps}{2}.
\end{equation*}
This implies that for $n\geq n_0$,
\begin{equation}
\label{cotainf}
\|u^*(m-1|k_n)-P_ru^*(m-1|k_n)\|\geq c-\eps.
\end{equation}
Indeed, applying the triangle inequality and using the fact that $(P_r v-P_r w)_{\perp}=\bf{0}$ for all $v, w\in \R^{n_u}$, we obtain
\begin{align*}
\|u^*(m-1|k_n)-P_ru^*(m-1|k_n)\|&=\|(u^*(m-1|k_n)-P_ru^*(m-1|k_n))_{\perp}\|\nonumber\\
&\geq \|(u^*(0|k_n-1)-P_ru^*(m-1|k_n))_{\perp}\|\\
&\phantom{***}-\|(u^*(m-1|k_n)-u^*(0|k_n-1))_{\perp}\|\\
&= \|u^*(0|k_n-1)-P_ru^*(0|k_n-1)\|\\
&\phantom{***}-\|(u^*(m-1|k_n)-u^*(0|k_n-1))_{\perp}\|\\
&\geq c-\frac{\eps}{2}-\frac{\eps}{2}\\
&=c-\eps.
\end{align*}
Also, there exists $k_0$ such that for all $k\geq k_0$ we have
\begin{equation}
\label{cotasup}
\|u^*(0|k-1)-P_r u^*(0|k-1)\|\leq c+\eps.
\end{equation}
Then for some $\alpha\in (0,1]$ to be chosen later, consider the following strategy at time $k$ such that $k\geq k_0$ and $k=k_n$ for some $n\geq n_0$.
$$
\Delta\tilde{u}(0|k)=\alpha (\Pi_r(u^*(0|k-1))-u^*(0|k-1))
$$
where $\Pi_rx$ is the point in $\mathbb{U}\cap U_r$ such that $\|x-\Pi_r x\|=\text{dist}(x,\mathbb{U}\cap U_r)$,
$\Delta\tilde{u}(j|k)=0$ for $j\geq 1$ and $\tilde{\delta}_k$ given by
$$
\tilde{\delta}_k=e^*_s(0|k-1)+D_0\Delta\tilde{u}(0|k).
$$
Since $\mathbb{U}$ is convex, we can choose $\alpha$ small enough such that this strategy is indeed feasible. We now compute the cost of this strategy and compare it to the optimal cost. First, observe that 
\begin{align*}
\tilde{\delta}_k&=e_s^*(0|k-1)+D_0\Delta\tilde{u}(0|k)\\
&=D_0u^*(0|k-1)-r+\alpha D_0[\Pi_ru^*(0|k-1)-u^*(0|k-1)]\\
&=(1-\alpha)D_0u^*(0|k-1)-r +\alpha r\\
&=(1-\alpha)(D_0u^*(0|k-1)-r)\\
&=(1-\alpha)D_0[u^*(0|k-1)-\Pi_ru^*(0|k-1)].
\end{align*}
Thus, using (\ref{cotasup}), we obtain that
\begin{align*}
\|\tilde{\delta}_k\|_S^2&=(1-\alpha)^2\|D_0[u^*(0|k-1)-\Pi_ru^*(0|k-1)]\|_S^2\\
&= (1-\alpha)^2\|D_0[u^*(0|k-1)-\Pi_ru^*(0|k-1)]_\perp\|_S^2\\
&=(1-\alpha)^2\beta\|[u^*(0|k-1)-\Pi_ru^*(0|k-1)]_\perp\|^2\\
&=(1-\alpha)^2\beta\|u^*(0|k-1)-P_ru^*(0|k-1)\|^2\\
&\leq (1-\alpha)^2\beta(c+\eps)^2.
\end{align*}
On the other hand, we have
\begin{align*}
\delta_k^*&=e_s^*(0|k-1)+D_0\sum_{j=0}^{m-1}\Delta u^*(j|k)\\
&=D_0u^*(m-1|k)-r\\
&=D_0[u^*(m-1|k)-\Pi_ru^*(m-1|k)]_\perp.
\end{align*}
Hence, using (\ref{cotainf}), we obtain that
\begin{align*}
\|\delta_k^*\|_S^2&=\|D_0[u^*(m-1|k)-\Pi_ru^*(m-1|k)]_\perp\|_S^2\\
&=\beta\|[u^*(m-1|k)-\Pi_ru^*(m-1|k)]_\perp\|^2\\
&=\beta\|u^*(m-1|k)-P_ru^*(m-1|k)\|^2\\
&\geq \beta(c-\eps)^2.
\end{align*}
After some elementary computation, we obtain that 
\begin{align*}
V^*_k-\tilde{V}_k&\geq \|\delta_k^*\|_S^2-\|\tilde{\delta}_k\|_S^2-\|\Delta \tilde{u}(0|k)\|_Z^2\\ 
&\phantom{**}-2\Big(\sum_{j=0}^{m-1}\langle\Psi F^{j+1}x_d^*(0|k-1),\Psi  F^jD_d\Delta\tilde{u}(0|k)\rangle_Q+\langle F^mx_d^*(0|k-1),  F^{m-1}D_d\Delta\tilde{u}(0|k) \rangle_{\bar{Q}}\Big)
\end{align*}
where 
\begin{align}
\label{DefG}
Z&=R+\sum_{j=0}^{m-1}(D_d)^T(F^j)^T\Psi^T Q\Psi F^jD_d+D_d^T(F^{m-1})^T\bar{Q}F^{m-1}D_d\nonumber\\
&=:R+D_d^TGD_d.
\end{align}
Using the Cauchy-Schwarz inequality, we have that
\begin{align*}
V^*_k-\tilde{V}_k&\geq \|\delta_k^*\|_S^2-\|\tilde{\delta}_k\|_S^2-\|\Delta \tilde{u}(0|k)\|_Z^2\\ 
&\phantom{**}-2\Big(\sum_{j=0}^{m-1}\|\Psi F^{j+1}x_d^*(0|k-1)\|_Q\|\Psi  F^jD_d\Delta\tilde{u}(0|k)\|_Q\\
&\;\;\;\;\;\;\;\;\;\;\;\;\;\;\;\;\;\;\;\;\;\;+\| F^mx_d^*(0|k-1)\|_{\bar{Q}}\| F^{m-1}D_d\Delta\tilde{u}(0|k) \|_{\bar{Q}}\Big)\\
&\geq \|\delta_k^*\|_S^2-\|\tilde{\delta}_k\|_S^2-\|\Delta \tilde{u}(0|k)\|_Z^2-C_1\| x_d^*(0|k-1)\|\|\Delta\tilde{u}(0|k)\|
\end{align*}
for some positive constant $C_1$ that depends on $m, \Psi, F, D_d$ and $Q$.
Thus, we obtain that
\begin{equation*}
V^*_k-\tilde{V}_k\geq \|\delta_k^*\|_S^2-\|\tilde{\delta}_k\|_S^2-C_2\|\Delta \tilde{u}(0|k)\|(\|\Delta\tilde{u}(0|k)\|+\| x_d^*(0|k-1)\|)
\end{equation*}
for some positive constant $C_2$ that depends on $m, R, \Psi, F, D_d$ and $Q$.

Now let $\theta_x$ be the angle between $P_rx -x$ and $\Pi_r x-x$ for $x\in \mathbb{U}\setminus \ker D_0$. Since $\mathbb{U}$ is a rectangle we have that $\varphi:=\inf_{x\in \mathbb{U}\setminus \ker D_0}|\cos \theta_x|>0$ and therefore $\|\Delta \tilde{u}(0|k)\|\leq \alpha\varphi^{-1}(c+\eps)$.
Now using Lemma \ref{lemma2}, we can choose $k$ large enough such that
$\|x_d^*(0|k-1)\|\leq \alpha \varphi^{-1}(c+\eps)$. Therefore, we obtain,

\begin{align*}
V^*_k-\tilde{V}_k&\geq \|\delta_k^*\|_S^2-\|\tilde{\delta}_k\|_S^2-C_3\alpha^2(c+\eps)^2\\
&\geq \beta(c-\eps)^2-(1-\alpha)^2\beta (c+\eps)^2-C_3\alpha^2(c+\eps)^2
\end{align*}
for some positive constant $C_3$ (see the Appendix for an explicit expression) that depends on the parameters of the model other than $S$.
Now, let us take $\eps<c/3$ small enough and $\alpha= 3\eps/(c+\eps)$ such that the strategy given by $\Delta \tilde {u}(0|k)$ is feasible. Consider $\beta> 6C_3$.
We can check that in this case $V^*_k-\tilde{V}_k>0$, which gives us the desired contradiction.
\end{proof}
\noindent
{\it Proof of Theorem} \ref{Maintheo}\\
From Proposition \ref{PropApprox} and Lemma \ref{lemma3}, we deduce that $\lim_{k\to \infty} \|\delta_k^*\|_S=0$. Indeed, observe that
\begin{align*}
 \|\delta_k^*\|_S&\leq \|D_0[u^*(0|k-1)-P_ru^*(0|k-1)]\|_S+\Big\|D_0\sum_{j=0}^{m-1}\Delta u^*(j|k)\Big\|_S\\
&= \sqrt{\beta}\| u^*(0|k-1)-P_r u^*(0|k-1)\|+ \sqrt{\beta}\Big\|\Big(\sum_{j=0}^{m-1}\Delta u^*(j|k)\Big)_\perp\Big\|\\
&\to 0+0
\end{align*}
as $k\to \infty$.

Finally, we will show that $\lim_{k\to \infty}V_k^*=0$. To this end, we first observe that using (\ref{Cost}), since $\lim_{k\to \infty} \|\delta_k^*\|_S=0$,
if $\limsup_{k\to \infty}\sum_{j=0}^{m-1}\|\Delta u^*(j|k)  \|_R^2=0$ then $\lim_{k\to \infty}V_k^*=0$. Thus, assume that $\lim_{k\to \infty}V_k^*>0$. This implies that $\limsup_{k\to \infty}\sum_{j=0}^{m-1}\|\Delta u^*(j|k)  \|_R^2=c^*>0$ and therefore there exists some subsequence $(k_n)_{n\geq 1}$ such that for all $n\geq 1$, $\sum_{j=0}^{m-1}\|\Delta u^*(j|k) \|_R^2\geq c^*/2$. Now consider $\eps<c^*/2$. For $n$ large enough using the strategy
$\Delta \tilde{u}(j|k_n)=0$, for all $0\leq j<m$ and $\tilde{\delta}_{k_n}=\delta^*_{k_n}-D_0\sum_{j=0}^{m-1}\Delta u^*(j|k_n)$  (we can check that it is indeed feasible) we obtain by Lemmas \ref{lemma2} and \ref{lemma3}
\begin{align*}
\tilde{V}_{k_n}&=\|Fx_d^*(0|k_n-1) \|_G^2 + \|\tilde{\delta}_{k_n}\|_S^2\\
&\leq \|Fx_d^*(0|k_n-1) \|_G^2 +\Big(\Big\|D_0\sum_{j=0}^{m-1}\Delta u^*(j|k_n)\Big\|_S + \|\delta^*_{k_n}\|_S\Big)^2\\
&\leq \eps\\
&<c^*/2\\
&\leq V^*_{k_n}
\end{align*}
where $G$ is from (\ref{DefG}). This gives us the desired contradiction.
\qed

\medskip\medskip

\noindent
{\it Proof of Theorem} \ref{Maintheo2}\\
Recall that the system is assumed to be initially in the steady state $u(0)={\bf 0}$, $x_s(0)={\bf 0}$, $x_d(0)={\bf 0}$, and consider the following feasible strategy for the control optimization problem at time step $k=1$,
\begin{align*}
	\Delta \tilde{u}(j|1) &= 0, \text{ for } j=0,1,\dots,m-1, \\
	\tilde{\delta}_{1} &= -r,
\end{align*}
which has an associated cost $\tilde{V}_1=\|r\|_S^2$. Since the above strategy is not necessarily the optimal one, and using Lemma \ref{lemma1}, we have that $V_k^* \leq V_1^* \leq \tilde{V}_1$ for all $k\geq 1$. But, for any $k\geq 1$, 
\begin{align*}
\|e^*(0|k)-\delta^*_k\|_Q^2 + \|\delta^*_k\|_S^2 \leq V_k^*
\end{align*}
and, on the other hand, by the parallelogram law,
\begin{align*}
\|e^*(0|k)\|_Q^2 \leq 2\big(\|e^*(0|k)-\delta^*_k\|_Q^2 + \|\delta^*_k\|_Q^2\big) \leq C_1\big(\|e^*(0|k)-\delta^*_k\|_Q^2 + \|\delta^*_k\|_S^2\big) 
\end{align*}
for some positive constant $C_1$. Thus we have that $\|e^*(0|k)\| \leq C_2\|r\|$ for some positive constant $C_2$, for all $k\geq 1$. This implies that, for any $\eta>0$, there exists $\varepsilon>0$ such that $\|e^*(0|k)\| \leq \eta$ for all $k\geq 1$, if $\|r\| \leq \varepsilon$. 
\qed

\section{Extended Infinite Horizon MPC with zone control}\label{Sec3}
\subsection{Formulation and results}
In this section, we consider the MPC with zone control and input targets (González and Odloak, 2009). This controller uses the OPOM model given by (\ref{model1}) and (\ref{model2}) for prediction. The input target $u_{des}\in\R^{n_u}$ is sent at each time step by the Real Time  Optimization (RTO) layer to the MPC layer. Unlike Odloak (2004), in this MPC the output set-point $y_{sp}\in\R^{n_y}$ is a variable of the optimization problem and is calculated at each time step. Here, the exact values of the output set-point are not important, as long as they remain inside a range with specified limits. As in Section \ref{Sec2}, for $j=0,1, \dots, m-1$, we define $\Delta u(j|k)$ as the $j$-th move of the input solution of the optimization problem (to be defined below) at time step $k$. For the other variables of the model, we will use a similar notation.

The MPC with zone control is based on the following cost function: for $k\geq 1$,
\begin{align*}
	V_k := \sum_{j=0}^{\infty} \big[y(j|k)-y_{sp,k}-\delta_{y,k}\big]^{T}Q_y\big[y(j|k)-y_{sp,k}-\delta_{y,k}\big] \\
	+ \sum_{j=0}^{\infty} \big[u(j|k)-u_{des}-\delta_{u,k}\big]^{T}Q_u\big[u(j|k)-u_{des}-\delta_{u,k}\big] \\
	+ \sum_{j=0}^{m-1} \Delta u(j|k)^{T}R\Delta u(j|k) + \delta_{y,k}^TS_y\delta_{y,k} + \delta_{u,k}^TS_u\delta_{u,k},
\end{align*}
where $\delta_{y,k}\in\R^{n_y}$ and $\delta_{u,k}\in\R^{n_u}$ are \textit{slack} vectors, and $Q_y\in \R^{n_y\times n_y}$, $Q_u\in \R^{n_u\times n_u}$, $R\in \R^{n_u\times n_u}$, $S_y\in \R^{n_y\times n_y}$ and $S_u\in \R^{n_u\times n_u}$ are positive definite weighting matrices. To prevent the cost from being unbounded, as discussed in González and Odloak (2009), we impose the terminal constraints
\begin{align}
	x_s(m-1|k) - y_{sp,k} -\delta_{y,k} = 0, ~~k\geq 1,
	\label{constx}
\end{align}
and
\begin{align}
	u(m-1|k) - u_{des} - \delta_{u,k} = 0, ~~k\geq 1.
	\label{constu}
\end{align}

In this section we also work under Assumption \ref{Assump1}, which for convenience we rename as

\begin{assump}
	The system is stable, that is, the spectral radius of the matrix $F$ is strictly smaller than $1$.
	\label{Assump4}
\end{assump}

Under Assumption \ref{Assump4} and constraints (\ref{constx}) and (\ref{constu}) we have that, for $k\geq 1$,
\begin{align*}
	V_k = \sum_{j=0}^{m-1} \big[y(j|k)-y_{sp,k}-\delta_{y,k}\big]^{T}Q_y\big[y(j|k)-y_{sp,k}-\delta_{y,k}\big] + x_d(m-1|k)^{T}\bar{Q}x_d(m-1|k) \\
	+ \sum_{j=0}^{m-1} \big[u(j|k)-u_{des}-\delta_{u,k}\big]^{T}Q_u\big[u(j|k)-u_{des}-\delta_{u,k}\big] + \sum_{j=0}^{m-1} \Delta u(j|k)^{T}R\Delta u(j|k) \\ 
	+ \,\delta_{y,k}^TS_y\delta_{y,k} + \delta_{u,k}^TS_u\delta_{u,k}
\end{align*}
or, in other words,
\begin{align*}
	V_k = \sum_{j=0}^{m-1} \big\|y(j|k)-y_{sp,k}-\delta_{y,k}\big\|_{Q_y}^2 + \big\|x_d(m-1|k)\big\|_{\bar{Q}}^2 + \sum_{j=0}^{m-1} \big\|u(j|k)-u_{des}-\delta_{u,k}\big\|_{Q_u}^2 \\
	+ \sum_{j=0}^{m-1} \big\|\Delta u(j|k)\big\|_R^2 + \big\|\delta_{y,k}\big\|_{S_y}^2 + \big\|\delta_{u,k}\big\|_{S_u}^2
\end{align*}
where
\begin{align*}
	\bar{Q} := \sum_{j=1}^{\infty} (F^j)^T\Psi^T Q_y \Psi F^j 
\end{align*}
is positive semidefinite and satisfies $\bar{Q} - F^{T}\bar{Q}F = F^{T}\Psi^{T}Q_y\Psi F$.

The control optimization problem of the MPC with zone control and input targets is the following: for all $k\geq 1$, 
\begin{align*}
	\min_{\Delta u_k, y_{sp,k}, \delta_{y,k}, \delta_{u,k}} V_k
\end{align*}
subject to (\ref{constx}), (\ref{constu}) and
\begin{align*}
	u(k-1) + \sum_{j=0}^{i} \Delta u(j|k) \in \mathbb{U}, \text{ for } 0\leq i < m,
\end{align*}
\begin{align*}
	\Delta u(j|k) \in \Delta\mathbb{U}, \text{ for } 0\leq j < m,
\end{align*}
\begin{align*}
	y_{sp,k} \in \mathbb{Y},
\end{align*}
where $\mathbb{U}$ and $\Delta\mathbb{U}$ are fixed rectangles in $\R^{n_u}$ and $\mathbb{Y}$ is a fixed rectangle in $\R^{n_y}$, all of them containing the corresponding origin, and
$$ \Delta u_k = \big[\Delta u(0|k)^T ~ \Delta u(1|k)^T ~ \cdots ~ \Delta u(m-1|k)^T\big]^T \in \R^{mn_u}. $$
Here we assume that, at time $0$, the system is in the steady state $u(0)=\textbf{0}$, $x_s(0)=\textbf{0}$, $x_d(0)=\textbf{0}$. Since $\mathbb{U}$, $\Delta\mathbb{U}$ and $\mathbb{Y}$ are convex sets and $R$ and $S_y$ are positive definite, the solution of the above optimization problem is unique. Let $\Delta u_k^*$, $y_{sp,k}^*$, $\delta_{y,k}^*$, $\delta_{u,k}^*$ be the solution of the control optimization problem at time step $k$ and let $V_k^*$ be the corresponding cost,
\begin{align*}
	V_k^* = \sum_{j=0}^{m-1} \big\|y^*(j|k)-y_{sp,k}^*-\delta_{y,k}^*\big\|_{Q_y}^2 + \big\|x_d^*(m-1|k)\big\|_{\bar{Q}}^2 + \sum_{j=0}^{m-1} \big\|u^*(j|k)-u_{des}-\delta_{u,k}^*\big\|_{Q_u}^2 \\
	+ \sum_{j=0}^{m-1} \big\|\Delta u^*(j|k)\big\|_R^2 + \big\|\delta_{y,k}^*\big\|_{S_y}^2 + \big\|\delta_{u,k}^*\big\|_{S_u}^2,
\end{align*}
where $u^*$ is obtained from $\Delta u_k^*$, and $y^*$, $x_s^*$ and $x_d^*$ are obtained from $\Delta u_k^*$, (\ref{model1}) and (\ref{model2}).

In the following, we also need the

\begin{assump}
	The input target $u_{des}$ is such that $u_{des}\in\mathbb{U}$ and $D_0u_{des} \in \mathbb{Y}$. 
	\label{Assump5}
\end{assump}

We finally state our results about convergence and stability of this second controller.

\begin{thm}[Convergence]
\label{Maintheo3}
	Under Assumptions \ref{Assump4} and \ref{Assump5}, we can choose $S_u$ large enough such that
	$$ \displaystyle\lim_{k\to\infty}V_k^*=0. $$
\end{thm}

\begin{thm}[Stability]
\label{Maintheo4}
	Under Assumptions \ref{Assump4} and \ref{Assump5}, the controller is stable, that is, for any $\eta>0$, there exists $\varepsilon>0$ such that $\|u_{des}\| \leq \varepsilon$ implies that 
	$$ \bigg\|\begin{bmatrix} y^*(0|k) - y^*_{sp,k} \\ u^*(0|k) - u_{des} \end{bmatrix}\bigg\| \leq \eta \,\text{ for all }\, k\geq 1. $$
\end{thm}

\subsection{Proofs of Theorems \ref{Maintheo3} and \ref{Maintheo4}}

We start with several technical results that will be useful to prove Theorems \ref{Maintheo3} and \ref{Maintheo4}. The following lemma was established in González and Odloak (2009), but for the sake of completeness we present its proof.

\begin{lem}
	Under Assumption \ref{Assump4}, the sequence $(V_k^*)_{k\geq 1}$ is non-increasing.
	\label{lemma4}
\end{lem}

\begin{proof}
	Let $\Delta \tilde{u}_{k+1}$, $\tilde{y}_{sp,k+1}$, $\tilde{\delta}_{y,k+1}$, $\tilde{\delta}_{u,k+1}$ be the following strategy for the control optimization problem at time step $k+1$:
	\begin{align*}
		\Delta \tilde{u}(j|k+1) &= \Delta u^*(j+1|k), \text{ for } j=0,1,\dots,m-2, \\
		\Delta \tilde{u}(m-1|k+1) &= 0, \\
		\tilde{y}_{sp,k+1} &= y_{sp,k}^*, \\
		\tilde{\delta}_{y,k+1} &= \delta_{y,k}^*, \\
		\tilde{\delta}_{u,k+1} &= \delta_{u,k}^*.
	\end{align*}
	 We can check that the above strategy is feasible for the control optimization problem at time step $k+1$, and its cost is
	\begin{align*}
		\tilde{V}_{k+1} = \sum_{j=0}^{m-1} \big\|\tilde{y}(j|k+1)-y_{sp,k}^*-\delta_{y,k}^*\big\|_{Q_y}^2 + \big\|\tilde{x}_d(m-1|k+1)\big\|_{\bar{Q}}^2 + \sum_{j=0}^{m-1} \big\|\tilde{u}(j|k+1)-u_{des}-\delta_{u,k}^*\big\|_{Q_u}^2 \\
		+ \sum_{j=1}^{m-1} \big\|\Delta u^*(j|k)\big\|_R^2 + \big\|\delta_{y,k}^*\big\|_{S_y}^2 + \big\|\delta_{u,k}^*\big\|_{S_u}^2.
	\end{align*}
	
	Now, since $\tilde{y}(j|k+1) = y^*(j+1|k)$ and $\tilde{u}(j|k+1) = u^*(j+1|k)$ for $0\leq j \leq m-2$, we have that 
	\begin{align*}
		\tilde{V}_{k+1} - V_k^* = \big\|\tilde{y}(m-1|k+1)-y_{sp,k}^*-\delta_{y,k}^*\big\|_{Q_y}^2 - \big\|y^*(0|k)-y_{sp,k}^*-\delta_{y,k}^*\big\|_{Q_y}^2 \\
		+ \big\|\tilde{u}(m-1|k+1)-u_{des}-\delta_{u,k}^*\big\|_{Q_u}^2 - \big\|u^*(0|k)-u_{des}-\delta_{u,k}^*\big\|_{Q_u}^2 \\
		+ \big\|\tilde{x}_d(m-1|k+1)\big\|_{\bar{Q}}^2 - \big\|x_d^*(m-1|k)\big\|_{\bar{Q}}^2 - \big\|\Delta u^*(0|k)\big\|_R^2.
	\end{align*}
	But observe that $\tilde{x}_s(m-1|k+1) = x_s^*(m-1|k)$ and then
	\begin{align*}
		\tilde{y}(m-1|k+1)-y_{sp,k}^*-\delta_{y,k}^* = \Psi\tilde{x}_d(m-1|k+1).
	\end{align*}
	On the other hand, we have that $\tilde{x}_d(m-1|k+1) =  F x_d^*(m-1|k)$, so that
	\begin{align*}
		\big\|\tilde{y}(m-1|k+1)-y_{sp,k}^*-\delta_{y,k}^*\big\|_{Q_y}^2 + \big\|\tilde{x}_d(m-1|k+1)\big\|_{\bar{Q}}^2
		= \big\|x_d^*(m-1|k)\big\|_{\bar{Q}}^2.
	\end{align*}
	Moreover, note that $\tilde{u}(m-1|k+1) = u^*(m-1|k)$ and then $\tilde{u}(m-1|k+1) - u_{des} - \delta_{u,k}^* =  0$. Thus, we deduce that
	\begin{align*}
		\tilde{V}_{k+1} - V_k^* = -\big\|y^*(0|k)-y_{sp,k}^*-\delta_{y,k}^*\big\|_{Q_y}^2 - \big\|u^*(0|k)-u_{des}-\delta_{u,k}^*\big\|_{Q_u}^2 - \big\|\Delta u^*(0|k)\big\|_R^2 \leq 0.
	\end{align*}
	
	Finally, since the strategy $\Delta \tilde{u}_{k+1}$, $\tilde{y}_{sp,k+1}$, $\tilde{\delta}_{y,k+1}$, $\tilde{\delta}_{u,k+1}$ at time step $k+1$ is not necessarily the optimal one, we have that $V_{k+1}^* \leq \tilde{V}_{k+1}$ and then we conclude that the sequence $(V_k^*)_{k\geq 1}$ is non-increasing.
\end{proof}

\begin{lem}
	Under Assumption \ref{Assump4}, we have that
	\begin{equation*}
		\sum_{j=0}^{m-1} \Delta u^*(j|k) \to 0 \,\text{ as } k\to\infty ~~~\text{ and }~~~ x_d^*(0|k) \to 0 \,\text{ as } k\to\infty. 
	\end{equation*}
	\label{lemma5}
\end{lem}

\begin{proof}
	From the proof of the last lemma, we have that
	\begin{align*}
		V_{k+1}^* - V_k^* \leq - \big\|y^*(0|k)-y_{sp,k}^*-\delta_{y,k}^*\big\|_{Q_y}^2 - \big\|u^*(0|k)-u_{des}-\delta_{u,k}^*\big\|_{Q_u}^2 - \big\|\Delta u^*(0|k)\big\|_R^2.
	\end{align*}
	Since the sequence $(V_k^*)_{k\geq 1}$ is non-increasing and non-negative, we deduce that it converges and then
	\begin{align}
		\label{Dutozero}
		\big\|u^*(0|k)-u_{des}-\delta_{u,k}^*\big\|_{Q_u}^2 \xrightarrow{k\to\infty} 0 
		~~~\text{ and }~~~ \big\|\Delta u^*(0|k)\big\|_R^2 \xrightarrow{k\to\infty} 0.
	\end{align}
	Using (\ref{constu}), we have that
	\begin{align*}
		u^*(0|k)-u_{des}-\delta_{u,k}^* = -\sum_{j=1}^{m-1}\Delta u^*(j|k).
	\end{align*}
	Thus, by (\ref{Dutozero}), we obtain that
	\begin{align*}
		\sum_{j=0}^{m-1}\Delta u^*(j|k) \xrightarrow{k\to\infty} 0.
	\end{align*}

	Now, observe that
	\begin{align*}
		x_d^*(0|k+p) = F^{p} x_d^*(0|k) + \sum_{j=0}^{p-1} F^j D_d \Delta u^*(0|k+p-j), ~\text{ for } p\geq 1.
	\end{align*}
	Under Assumption \ref{Assump4}, for any $\varepsilon>0$, there exists $k_0\geq 1$ such that
	\begin{align*}
		\sup_{p\geq 1}\big\|x_d^*(0|k_0+p) - F^{p} x_d^*(0|k_0)\big\| < \varepsilon/2,
	\end{align*}
	and there exists $p_0\geq 1$ such that, for all $p\geq p_0$, $\|F^{p} x_d^*(0|k_0)\| < \varepsilon/2$.
	Finally, for all $p\geq p_0$,
	\begin{align*}
		\big\|x_d^*(0|k_0+p)\big\| \leq \big\|F^{p} x_d^*(0|k_0)\big\| + \sup_{p\geq 1}\big\|x_d^*(0|k_0+p) - F^{p} x_d^*(0|k_0)\big\| < \varepsilon,
	\end{align*}
	that is, $x_d^*(0|k) \to 0$ as $k\to\infty$.
\end{proof}

\begin{lem}
	Under Assumption \ref{Assump4}, we have that
	\begin{equation*}
		\lim_{k\to\infty}V_k^* = \lim_{k\to\infty} \big(\|\delta_{y,k}^*\|_{S_y}^2 + \|\delta_{u,k}^*\|_{S_u}^2\big). 
	\end{equation*}
	\label{lemma6}
\end{lem}
\begin{proof}
	By Lemma \ref{lemma5}, we can choose $k$ large enough so that
	\begin{align*}
		\sum_{j=0}^{m-1} \Delta u^*(j|k) \in \Delta\mathbb{U}.
	\end{align*}
	For such a $k$, let $\Delta \tilde{u}_{k}$, $\tilde{y}_{sp,k}$, $\tilde{\delta}_{y,k}$, $\tilde{\delta}_{u,k}$ be such that
	\begin{align*}
		\Delta \tilde{u}(0|k) &= \sum_{j=0}^{m-1}\Delta u^*(j|k), \\
		\Delta \tilde{u}(j|k) &= 0, \text{ for } 1\leq j <m, \\
		\tilde{y}_{sp,k} &= y_{sp,k}^*, \\
		\tilde{\delta}_{y,k} &= \delta_{y,k}^*, \\
		\tilde{\delta}_{u,k} &= \delta_{u,k}^*,
	\end{align*}
	and observe that this is a feasible strategy for the control optimization problem at time step $k$. Indeed, we have that
	\begin{align*}
		x_s^*(0|k-1) + D_0\sum_{j=0}^{m-1}\Delta \tilde{u}(j|k) - \tilde{y}_{sp,k} - \tilde{\delta}_{y,k} = x_s^*(0|k-1) + D_0\sum_{j=0}^{m-1}\Delta u^*(j|k) - y_{sp,k}^* - \delta_{y,k}^* = 0
	\end{align*}
	and
	\begin{align*}
		u^*(k-1) + \sum_{j=0}^{m-1}\Delta \tilde{u}(j|k) - u_{des} - \tilde{\delta}_{u,k} = u^*(k-1) + \sum_{j=0}^{m-1}\Delta u^*(j|k) - u_{des} - \delta_{u,k}^* = 0,
	\end{align*}
	so that (\ref{constx}) and (\ref{constu}) are satisfied, and we also have that
	\begin{align*}
		u^*(k-1) + \sum_{i=0}^{j} \Delta \tilde{u}(i|k) = u^*(k-1) + \sum_{i=0}^{m-1} \Delta u^*(i|k) \in \mathbb{U}
	\end{align*}
	for $0\leq j < m$. Moreover, note that $\tilde{y}_{sp,k} \in \mathbb{Y}$ and $\Delta \tilde{u}(j|k) \in \Delta\mathbb{U}$ for $0\leq j < m$.
	
	Now, the cost corresponding to the above strategy is
	\begin{align*}
		\tilde{V}_k = \sum_{j=0}^{m-1} \big\|\tilde{y}(j|k)-y_{sp,k}^*-\delta_{y,k}^*\big\|_{Q_y}^2 + \big\|\tilde{x}_d(m-1|k)\big\|_{\bar{Q}}^2 + \sum_{j=0}^{m-1} \big\|\tilde{u}(j|k)-u_{des}-\delta_{u,k}^*\big\|_{Q_u}^2 \\
		+ \Big\|\sum_{j=0}^{m-1} \Delta u^*(j|k)\Big\|_R^2 + \big\|\delta_{y,k}^*\big\|_{S_y}^2 + \big\|\delta_{u,k}^*\big\|_{S_u}^2
	\end{align*}
	and, for $0\leq j < m$, we have that $\tilde{u}(j|k)-u_{des}-\delta_{u,k}^* = 0$ and
	\begin{align*}
		\tilde{y}(j|k)-y_{sp,k}^*-\delta_{y,k}^* = \Psi F^j \Big(Fx_d^*(0|k-1) + D_d\sum_{i=0}^{m-1}\Delta u^*(i|k)\Big).
	\end{align*}
	Moreover, we also have that
	\begin{align*}
		\tilde{x}_d(m-1|k) = F^{m-1} \Big(Fx_d^*(0|k-1) + D_d\sum_{j=0}^{m-1}\Delta u^*(j|k)\Big)
	\end{align*}
	and then we obtain that
	\begin{align*}
		\tilde{V}_{k} &= \sum_{j=0}^{m-1} \Big\|\Psi F^j \Big(Fx_d^*(0|k-1) + D_d\sum_{i=0}^{m-1}\Delta u^*(i|k)\Big)\Big\|_{Q_y}^2 \\ 
		& ~~~~+ \Big\|F^{m-1} \Big(Fx_d^*(0|k-1) + D_d\sum_{j=0}^{m-1}\Delta u^*(j|k)\Big)\Big\|_{\bar{Q}}^2 + \Big\|\sum_{j=0}^{m-1}\Delta u^*(j|k)\Big\|_R^2 + \big\|\delta_{y,k}^*\big\|_{S_y}^2 + \big\|\delta_{u,k}^*\big\|_{S_u}^2 \\
		&= \Big\|Fx_d^*(0|k-1) + D_d\sum_{j=0}^{m-1}\Delta u^*(j|k)\Big\|_G^2 + \Big\|\sum_{j=0}^{m-1}\Delta u^*(j|k)\Big\|_R^2 + \big\|\delta_{y,k}^*\big\|_{S_y}^2 + \big\|\delta_{u,k}^*\big\|_{S_u}^2
	\end{align*}
	where 
	\begin{align*}
		G = \sum_{j=0}^{m-1} (\Psi F^j)^TQ_y(\Psi F^j) + (F^{m-1})^T\bar{Q}F^{m-1} = \sum_{j=0}^{\infty} (\Psi F^j)^TQ_y(\Psi F^j) = \Psi^TQ_y\Psi + \bar{Q}.
	\end{align*}
	Hence, by Lemma \ref{lemma5}, we deduce that
	\begin{align*}
		\tilde{V}_{k} - \big(\|\delta_{y,k}^*\|_{S_y}^2 + \|\delta_{u,k}^*\|_{S_u}^2\big) \xrightarrow{k\to\infty} 0.
	\end{align*}
	On the other hand, note that $V_k^* \leq \tilde{V}_k$ and $V_k^* - (\|\delta_{y,k}^*\|_{S_y}^2 + \|\delta_{u,k}^*\|_{S_u}^2) \geq 0$, therefore
	\begin{align*}
		V_k^* - \big(\|\delta_{y,k}^*\|_{S_y}^2 + \|\delta_{u,k}^*\|_{S_u}^2\big) \xrightarrow{k\to\infty} 0.
	\end{align*}
	Finally, since the sequence $(V_k^*)_{k\geq 1}$ converges, we conclude that
	\begin{equation*}
		\lim_{k\to\infty}V_k^* = \lim_{k\to\infty} \big(\|\delta_{y,k}^*\|_{S_y}^2 + \|\delta_{u,k}^*\|_{S_u}^2\big). 
	\end{equation*}
\end{proof}

As a byproduct of the last lemma, we have the following
\begin{cor}
	Under the same assumption of Lemma \ref{lemma6}, we have that
	\begin{align*}
		\sum_{j=0}^{m-1} \big\|\Delta u^*(j|k)\big\|_R^2 \to 0 ~\text{ as } k\to\infty. 
	\end{align*}
	\label{cor1}
\end{cor}

\noindent
{\it Proof of Theorem} \ref{Maintheo3}\\
	Assume that $S_u>H+I_{n_u}$, where $I_{n_u}$ is the identity matrix of dimension $n_u$ and
	\begin{align*}
		H = (m-1)D_0^TQ_yD_0 + (\Psi D_d)^TQ_y(\Psi D_d) + D_d^T\bar{Q}D_d + (m-1)Q_u + R,
	\end{align*}
	and suppose that $\lim_{k\to\infty}V_k^*>0$.
	
	By Corollary \ref{cor1}, we can choose a large enough $k$ such that
	\begin{align}
		\Delta u^*(m-1|k) \in \text{int}\,\Delta\mathbb{U}.
		\label{largek}
	\end{align}
	For such a $k$, let $\alpha=\alpha(k)\in (0,1)$ be such that
	\begin{align*}
		\Delta u^*(m-1|k) - (1-\alpha) \delta_{u,k}^* \in \Delta\mathbb{U}
	\end{align*}
	and consider the following strategy for the control optimization problem at time step $k$:
	\begin{align*}
		\Delta \tilde{u}(j|k) &= \Delta u^*(j|k), \text{ for } 0\leq j <m-1, \\
		\Delta \tilde{u}(m-1|k) &= \Delta u^*(m-1|k) - (1-\alpha) \delta_{u,k}^*,  \\
		\tilde{y}_{sp,k} &= \alpha y_{sp,k}^* + (1-\alpha)D_0u_{des}, \\
		\tilde{\delta}_{y,k} &= \alpha\delta_{y,k}^*, \\
		\tilde{\delta}_{u,k} &= \alpha\delta_{u,k}^*.
	\end{align*}
	It is straightforward to check that this strategy satisfies (\ref{constu}) and, from the fact that the solution of the control optimization problem at time step $k$ satisfies (\ref{constx}) and (\ref{constu}), we deduce that 
	\begin{align*}
		D_0(u_{des}+\delta_{u,k}^*) = y_{sp,k}^* + \delta_{y,k}^*
	\end{align*}
	which can be used to check that the above strategy also satisfies (\ref{constx}). Moreover, we have that $\Delta \tilde{u}(j|k) \in \Delta\mathbb{U}$, for $0\leq j < m$, and
	\begin{align*}
		u^*(k-1) + \sum_{i=0}^{j} \Delta \tilde{u}(i|k) \in \mathbb{U}, \text{ for } 0\leq j < m-1.
	\end{align*}
	Using the fact that $\delta_{u,k}^* = u^*(m-1|k) - u_{des}$, which follows from (\ref{constu}), we also have that
	\begin{align*}
		u^*(k-1) + \sum_{i=0}^{m-1} \Delta \tilde{u}(i|k) = \alpha u^*(m-1|k) + (1-\alpha) u_{des} \in\mathbb{U},
	\end{align*}
	since both $u^*(m-1|k) \in\mathbb{U}$ and $u_{des} \in\mathbb{U}$. Also, note that $\tilde{y}_{sp,k}\in\mathbb{Y}$ since both $y_{sp,k}^*\in\mathbb{Y}$ and $D_0u_{des}\in\mathbb{Y}$. Thus, the above strategy is indeed feasible at time step $k$, and the corresponding cost is
	\begin{align*}
		\tilde{V}_k = \sum_{j=0}^{m-1} \big\|y^*(j|k)-y_{sp,k}^*-\delta_{y,k}^*\big\|_{Q_y}^2 + (m-1)(1-\alpha)^2 \big\|D_0\delta_{u,k}^*\big\|_{Q_y}^2 + (1-\alpha)^2 \big\|\Psi D_d \delta_{u,k}^*\big\|_{Q_y}^2 \\
		+ 2(1-\alpha)\Big(\sum_{j=0}^{m-2} \big\langle y^*(j|k)-y_{sp,k}^*-\delta_{y,k}^*, D_0\delta_{u,k}^*\big\rangle_{Q_y} - \big\langle \Psi x_d^*(m-1|k), \Psi D_d\delta_{u,k}^*\big\rangle_{Q_y} \Big) \\
		+ \big\|x_d^*(m-1|k)\big\|_{\bar{Q}}^2 + (1-\alpha)^2 \big\|D_d\delta_{u,k}^*\big\|_{\bar{Q}}^2 - 2(1-\alpha) \big\langle x_d^*(m-1|k), D_d\delta_{u,k}^*\big\rangle_{\bar{Q}} \\
		+ \sum_{j=0}^{m-1} \big\|u^*(j|k)-u_{des}-\delta_{u,k}^*\big\|_{Q_u}^2 + (m-1)(1-\alpha)^2 \big\|\delta_{u,k}^*\big\|_{Q_u}^2 \\
		+ 2(1-\alpha) \sum_{j=0}^{m-2} \big\langle u^*(j|k)-u_{des}-\delta_{u,k}^*, \delta_{u,k}^*\big\rangle_{Q_u} \\
		+ \sum_{j=0}^{m-1} \big\|\Delta u^*(j|k)\big\|_R^2 + (1-\alpha)^2 \big\|\delta_{u,k}^*\big\|_{R}^2 - 2(1-\alpha) \big\langle \Delta u^*(m-1|k), \delta_{u,k}^*\big\rangle_{R} \\
		+ \alpha^2 \big\|\delta_{y,k}^*\big\|_{S_y}^2 + \alpha^2 \big\|\delta_{u,k}^*\big\|_{S_u}^2.
	\end{align*}
	Hence, we have that
	\begin{align*}
		\frac{\tilde{V}_k - V_k^*}{1-\alpha} = 2\Big(\sum_{j=0}^{m-2} \big\langle y^*(j|k)-y_{sp,k}^*-\delta_{y,k}^*, D_0\delta_{u,k}^*\big\rangle_{Q_y} - \big\langle \Psi x_d^*(m-1|k), \Psi D_d\delta_{u,k}^*\big\rangle_{Q_y} \\
		- \big\langle x_d^*(m-1|k), D_d\delta_{u,k}^*\big\rangle_{\bar{Q}} + \sum_{j=0}^{m-2} \big\langle u^*(j|k)-u_{des}-\delta_{u,k}^*, \delta_{u,k}^*\big\rangle_{Q_u} \\
		- \big\langle \Delta u^*(m-1|k), \delta_{u,k}^*\big\rangle_{R} \Big) \\
		- (1+\alpha) \big\|\delta_{y,k}^*\big\|_{S_y}^2 + (1-\alpha) \big\|\delta_{u,k}^*\big\|_{H}^2 - (1+\alpha) \big\|\delta_{u,k}^*\big\|_{S_u}^2.
	\end{align*}
	
	Now, if $\limsup_{k\to\infty}\|\delta_{u,k}^*\|_{S_u}^2 >0$ then by equivalence of the norms, we have that\\ $\limsup_{k\to\infty}\|\delta_{u,k}^*\|^2 >0$ (where $\|\cdot\|$ is the euclidean norm) and therefore there exists $c>0$ such that $\|\delta_{u,k}^*\|^2 >c$ infinitely often. Since 
	\begin{align}
		x_d^*(m-1|k) = F^{m}x_d^*(0|k-1) + \sum_{i=0}^{m-1} F^{m-1-i} D_d\Delta u^*(i|k)
		\label{eq1}
	\end{align}
	and, for $0\leq j\leq m-2$,
	\begin{align}
		y^*(j|k)-y_{sp,k}^*-\delta_{y,k}^* = \Psi \big[F^{j+1}x_d^*(0|k-1) + \sum_{i=0}^{j} F^{j-i} D_d\Delta u^*(i|k)\big] -D_0\sum_{i=j+1}^{m-1} \Delta u^*(i|k)
		\label{eq2}
	\end{align}
	and also
	\begin{align}
		u^*(j|k)-u_{des}-\delta_{u,k}^* = - \sum_{i=j+1}^{m-1} \Delta u^*(i|k),
		\label{eq3}
	\end{align}
	by Lemma \ref{lemma5} and Corollary \ref{cor1}, for $0<\varepsilon<c$, there exists $k_0$ (large enough) satisfying (\ref{largek}) such that
	\begin{align*}
		\frac{\tilde{V}_{k_0} - V_{k_0}^*}{1-\alpha(k_0)}  < \varepsilon + \|\delta^*_{u,k_0}\|_{H}^2 -  \|\delta_{u,k_0}^*\|_{S_u}^2 < \varepsilon - \|\delta_{u,k_0}^*\|^2 < \varepsilon - c,
	\end{align*}
	since $S_u>H+I_{n_u}$. Finally, as $\varepsilon-c<0$, we conclude that $\tilde{V}_{k_0}<V_{k_0}^*$, which is a contradiction.
	
	On the other hand, if $\limsup_{k\to\infty}\|\delta_{u,k}^*\|_{S_u}^2 =0$ then, by Lemma \ref{lemma6}, we have that
	\begin{align*}
		\lim_{k\to\infty}\|\delta_{y,k}^*\|_{S_y}^2 = \lim_{k\to\infty} V_k^* >0,
	\end{align*}
	and using (\ref{eq1}), (\ref{eq2}), (\ref{eq3}), Lemma \ref{lemma5} and Corollary \ref{cor1} we deduce that, for $0<\varepsilon'<\lim_{k\to\infty} V_k^*$, there exists $k_1$ (large enough) satisfying (\ref{largek}) such that
	\begin{align*}
		\frac{\tilde{V}_{k_1} - V_{k_1}^*}{1-\alpha(k_1)}  < \frac{\varepsilon'}{2} - \|\delta_{y,k_1}^*\|_{S_y}^2 < \frac{\varepsilon'}{2} - \Big(\lim_{k\to\infty} V_k^* - \frac{\varepsilon'}{2}\Big) = \varepsilon' - \lim_{k\to\infty} V_k^* < 0
	\end{align*}
	and then $\tilde{V}_{k_1}<V_{k_1}^*$, which is a contradiction once again.
\qed

\medskip\medskip

\noindent
{\it Proof of Theorem} \ref{Maintheo4}\\
	Consider the following strategy for the control optimization problem at time step $k=1$,
	\begin{align*}
		\Delta \tilde{u}(j|1) &= 0, \text{ for } 0\leq j <m, \\
		\tilde{y}_{sp,1} &= D_0u_{des}, \\
		\tilde{\delta}_{y,1} &= -D_0u_{des}, \\
		\tilde{\delta}_{u,1} &= -u_{des}.
	\end{align*}
	Since the system is assumed to be initially in the steady state $u(0)={\bf 0}$, $x_s(0)={\bf 0}$, $x_d(0)={\bf 0}$, the above strategy is indeed feasible and it has an associated cost given by 
	$$ \tilde{V}_1=\|D_0u_{des}\|_{S_y}^2 + \|u_{des}\|_{S_u}^2. $$
	Moreover, since this strategy is not necessarily the optimal one, and using Lemma \ref{lemma4}, we have that $V_k^* \leq V_1^* \leq \tilde{V}_1$ for all $k\geq 1$. 
	
	Now, for any $k\geq 1$, we have that
	\begin{align*}
		\|y^*(0|k)-y^*_{sp,k}-\delta^*_{y,k}\|_{Q_y}^2 + \|\delta^*_{y,k}\|_{S_y}^2 + \|u^*(0|k)-u_{des}-\delta^*_{u,k}\|_{Q_u}^2 + \|\delta^*_{u,k}\|_{S_u}^2 \leq V_k^*.
	\end{align*}
	On the other hand,
	\begin{align*}
		\|y^*(0|k)-y^*_{sp,k}\|_{Q_y}^2 &\leq 2\big(\|y^*(0|k)-y^*_{sp,k}-\delta^*_{y,k}\|_{Q_y}^2 + \|\delta^*_{y,k}\|_{Q_y}^2\big) \\
		&\leq C_1\big(\|y^*(0|k)-y^*_{sp,k}-\delta^*_{y,k}\|_{Q_y}^2 + \|\delta^*_{y,k}\|_{S_y}^2\big) 
	\end{align*}
	for some positive constant $C_1$ and
	\begin{align*}
		\|u^*(0|k)-u_{des}\|_{Q_u}^2 &\leq 2\big(\|u^*(0|k)-u_{des}-\delta^*_{u,k}\|_{Q_u}^2 + \|\delta^*_{u,k}\|_{Q_u}^2\big) \\
		&\leq C_2\big(\|u^*(0|k)-u_{des}-\delta^*_{u,k}\|_{Q_u}^2 + \|\delta^*_{u,k}\|_{S_u}^2\big) 
	\end{align*}
	for some positive constant $C_2$. Thus, we obtain
	\begin{align*}
		\bigg\|\begin{bmatrix} y^*(0|k) - y^*_{sp,k} \\ u^*(0|k) - u_{des} \end{bmatrix}\bigg\| = \Big(\|y^*(0|k)-y^*_{sp,k}\|^2 + \|u^*(0|k)-u_{des}\|^2\Big)^{1/2} \leq C_3 \|u_{des}\|
	\end{align*}
	for some positive constant $C_3$, for all $k\geq 1$, which finally implies the result.
\qed

\section{Appendix}
In this appendix, we give an explicit expression for the constant $C_3$ that appears in the proof of Proposition \ref{PropApprox}.  For a positive semidefinite matrix $M$, let us define 
$$ \Gamma_M^2 := \sup_{\|x\|=1} x^T M x, $$
where $\|\cdot\|$ is the euclidean norm. It is well known that $\Gamma_M^2$ is the spectral radius of $M$.

Now, observe that
\begin{align*}
\sum_{j=0}^{m-1}\|\Psi F^{j+1}&x_d^*(0|k-1)\|_Q  \|\Psi  F^jD_d\Delta\tilde{u}(0|k)\|_Q +\| F^mx_d^*(0|k-1)\|_{\bar{Q}}\| F^{m-1}D_d\Delta\tilde{u}(0|k) \|_{\bar{Q}} \\
&\leq \Big(\sum_{j=0}^{m-1}\|\Psi F^{j+1}x_d^*(0|k-1)\|_Q +\| F^mx_d^*(0|k-1)\|_{\bar{Q}}\Big)\\ &\;\;\;\;\;\;\;\;\;\;\;\;\;\;\;\;\;\;\times\Big(\sum_{j=0}^{m-1}\|\Psi  F^jD_d\Delta\tilde{u}(0|k)\|_Q +\| F^{m-1}D_d\Delta\tilde{u}(0|k) \|_{\bar{Q}}\Big) \\
&= \|x_d^*(0|k-1)\|_{\bar{Q}} \|\Delta\tilde{u}(0|k)\|_{Z-R} \\
&\leq \Gamma_{\bar{Q}}\Gamma_{Z-R}\|x_d^*(0|k-1)\|\|\Delta\tilde{u}(0|k)\|,
\end{align*}
therefore we obtain that 
$$ C_3 = 2\varphi^{-2} \big[\Gamma_Z^2\vee (2\Gamma_{\bar{Q}}\Gamma_{Z-R})\big], $$
where $\varphi$ and $Z$ are defined in the proof of Proposition \ref{PropApprox}.

Finally, it is worth noting that in the case $D_0$  regular, the expression of $\hat{S}$ boils down to
$\hat{S}=(D_0^{-1})^TD_0^{-1}$ and $C_3=2 \big[\Gamma_Z^2\vee (2\Gamma_{\bar{Q}}\Gamma_{Z-R})\big]$ since in this case $\varphi=1$.

\section*{Acknowledgements}
L.A.~Alvarez thanks FAEPEX (2556/23) for finantial support. C.~Gallesco thanks The S\~ao Paulo Research Foundation, FAPESP (2023/07228-9) and FAEPEX (3073/23) for finantial support.

\end{document}